\documentclass[10pt]{article}

\usepackage{amsmath,amsfonts,amssymb,amsthm,graphicx}
\usepackage{enumerate}
\usepackage{etoolbox}
\usepackage{color}
\usepackage[colorlinks=true,citecolor=blue,pdfpagemode=UseNone,pdfstartview=FitH]{hyperref}

\usepackage[T2A]{fontenc}
\newenvironment{cyr}{}{}

\allowdisplaybreaks[4]

\newcommand{\noop}[1]{}

\newcommand{\R}{\mathbb{R}}    % real numbers
\newcommand{\N}{\mathbb{N}}    % natural numbers

\newcommand{\st}{:} % "such that"

\newcommand{\CCC}{\mathcal{C}}   % collection of sets
\newcommand{\UUU}{\mathfrak{U}}  % free ultrafilter

\theoremstyle{plain}
\newtheorem{theorem}{Theorem}

\newtheorem{proposition}[theorem]{Proposition}
\theoremstyle{definition}

\theoremstyle{remark}
\newtheorem{remark}[theorem]{Remark}

\newcommand{\abstr}%
  {Kolmogorov's \emph{Grundbegriffe der Wahrscheinlichkeitsrechnung}
  introduced the standard measure-theoretic formalization of probability
  but included only a two-page section
  about the relation of the mathematical theory of probability
  to the world of experience.
  This note discusses this brief non-mathematical section
  concentrating on what I regard
  as its two most controversial features:
  having two bridges between the mathematical theory and the world of experience
  instead of the standard one,
  and the possibility of observing a prespecified event of probability zero
  in a sequence of trials.}

\begin{document}

\title{The non-mathematical section of Kolmogorov's \emph{Grundbegriffe}}
\author{Vladimir Vovk}
\maketitle
\begin{abstract}
  \smallskip
  \abstr

  The version of this paper at \url{http://gtfp.net} (Working Paper 69)
  is updated most often.
\end{abstract}

\section{Introduction}

Kolmogorov's \emph{Grundbegriffe der Wahrscheinlichkeitsrechnung},
first published in German in 1933 and introducing
the standard measure-theoretic framework for probability,
is often praised for its philosophical neutrality
(see, e.g., \cite{Hajek:2023}),
but it also includes a two-page section about the philosophy of probability.
The title of this section, \cite[Sect.~I.2]{Kolmogoroff:1933},
is ``Das Verh\"altnis zur Erfahrungswelt'',
and Kolmogorov's footnote to this title
makes it clear that it is a variation on von Mises's section title
``Das Verh\"altnis der Theorie zur Erfahrungswelt''
in his book \cite{vonMises:1931}.
In this note I will refer to Kolmogorov's section
as the \emph{Verh\"altnis}.

The main topic of this note is unusual features of the \emph{Verh\"altnis}:
having two bridges connecting measure-theoretic probability with reality,
which Kolmogorov calls Principle~A and Principle~B,
and the possibility of observing a prespecified event of probability zero
\cite[Sect.~I.2, Remark~II]{Kolmogoroff:1933}.
Both unusual features are related to the possibility of combining
several experiments satisfying Kolmogorov's conditions, or some of the conditions,
into one that still satisfies them.
This will be covered in Sects.~\ref{sec:why-two} and~\ref{sec:Remark-II}.
The main finding of Sect.~\ref{sec:Remark-II}
is that Remark~II in \cite[Sect.~I.2]{Kolmogoroff:1933}
is consistent with Kolmogorov's Axioms I--V and inconsistent with Axiom~VI
(that of continuity,
equivalent to $\sigma$-additivity in the presence of the other axioms).

\section{Some history of the \emph{Grundbegriffe} and the \emph{Verh\"altnis}}

An important source of information about the \emph{Grundbegriffe}
is the volume \cite{Shiryaev:2003b} of correspondence
between Kolmogorov and Aleksandrov
(the title of the volume is a line from Kolmogorov's poem
devoted to Aleksandrov \cite[preface, p.~14]{Shiryaev:2003a}).
Kolmogorov started writing it on 26 October 1932
(letter to Aleksandrov of this date, \cite{Shiryaev:2003b})
and planned to work on it until January 1933 (letter of 8 November).
The preface is dated ``Easter 1933'' (which was 16 April 1933,
both in Russia and in the West).

The \emph{Grundbegriffe} was translated into Russian in 1936
by Kolmogorov's student Grigory Bavli
and then updated in 1974 by his other student, Albert Shiryaev.
The second Russian edition modernized Kolmogorov's notation,
and Shiryaev also extended the last three sections
\cite[Sects.~VI.3--VI.5]{Kolmogoroff:1933} by adding proofs
(following Kolmogorov's contemporary papers).
It seems very likely that Kolmogorov saw and approved the Russian translations.
The third Russian edition appeared in 1998,
after Kolmogorov's death in 1987;
it seems identical to the second edition of 1974
apart from a new one-page preface by Prokhorov and Shiryaev
and Shiryaev's essay on the history of probability
(later translated into English as \cite{Shiryaev:2019-local}).

An English translation of the \emph{Grundbegriffe} was published only in 1950,
although in his preface the translator, Nathan Morrison,
mentions an earlier English translation.
Morrison's translation used both the German original and the first Russian translation.
Another English translation of the \emph{Verh\"altnis}
appeared as \cite[Sect.~5.2.1]{Shafer/Vovk:2006SS};
I will use it when quoting (directly and indirectly) from the \emph{Verh\"altnis}
and will use Morrison's translation when quoting from the rest of the \emph{Grundbegriffe}.

One difference between the three Russian editions and the German original
is that the former merge the latter's Axioms~I and~II
into one axiom, Axiom~I.
The numbering of axioms used in this note is the one in the German original
(and the English translation);
there are five axioms (I--V) preceding the \emph{Verh\"altnis}
and one axiom (Axiom~VI of continuity) following it.

The earliest known draft of Chapter~I of the \emph{Grundbegriffe},
including the \emph{Verh\"altnis},
was prepared by Kolmogorov for his lectures in Dnipro.%
\footnote{%
  Dnipro was, and still is, Ukraine's fourth largest city
  (after Kharkiv, which was the capital at the time, Kyiv, and Odesa).
  It has been known under many names,
  including Katerynoslav, Dnipropetrovsk,
  and their Russian variations
  (Yekaterynoslav and Dnepropetrovsk).
  Kolmogorov and Aleksandrov refer to it as Dnepropetrovsk,
  but even in the Soviet times the city was widely known as Dnipro
  (or Dnepr in Russian).
  Since 2016 Dnipro is the official name.}
The \emph{Verh\"altnis} was copied from the lecture notes
(\begin{cyr}конспект\end{cyr})
of Kolmogorov's Dnipro lectures and greatly expanded,
according to Kolmogorov's plan
that he shared with Aleksandrov in the same letter
of 26 October 1932 \cite{Shiryaev:2003b}.
The whole of Chapter~I was to be largely identical
(\begin{cyr}довольно точно будет совпадать\end{cyr})
to the Dnipro lecture notes.
Kolmogorov had been visiting Dnipro twice a year
since Autumn 1931
\cite[p.~43]{Nikolsky:1983}, in Spring and Autumn,
and this continued, perhaps less regularly, until 1940
\cite[p.~113]{Nikolsky:2000}.
His last visit before starting the \emph{Grundbegriffe}
ended on 26 September 1932
(his letter to Aleksandrov of 24 September announces his departure),
and he visited it again during the \emph{Grundbegriffe} production,
lecturing on 4--10 March 1933 mainly on probability theory
(letter of 4 March).

\section{Summary of the \emph{Verh\"altnis}}

When discussing the \emph{Verh\"altnis},
I will use the notation used in the second and third Russian editions
(and the English translation in \cite[Sect.~5.2.1]{Shafer/Vovk:2006SS}).

The \emph{Verh\"altnis} consists of three parts:
a general scheme of the applications of probability,
including Principles~A and~B;
a frequentist derivation of the measure-theoretic axioms
introduced by this point (not including the axiom of continuity);
Remarks~I and~II.

In describing the way the theory of probability is applied to the real world,
Kolmogorov talks about a complex of conditions $\mathfrak{S}$
which allows an unlimited (unbeschr\"ankt, \begin{cyr}неограниченный\end{cyr}) number of repetitions.
The theory of probability is applicable when we may assume
that each event $A$ that does or does not occur under conditions $\mathfrak{S}$
is assigned a number $P(A)\in[0,1]$ with the following properties:
\begin{description}
\item[Principle A]
  One can be practically certain that,
  if $\mathfrak{S}$ is repeated a large number of times,
  $n$,
  and the event $A$ occurs $m$ times,
  then $m/n$ will differ only slightly from $P(A)$.
\item[Principle B]
  If $P(A)$ is very small,
  then one can be practically certain
  that the event $A$ will not occur
  on a single realization of $\mathfrak{S}$.
\end{description}
What is implicit in Principles~A and~B
is that the event $A$ should be prespecified,
i.e., chosen prior to the experiment.

\begin{remark}\label{rem:infinity}
  The title of Chapter~I
  (of which the \emph{Verh\"altnis} forms a part)
  is ``Elementary theory of probability'',
  and at the beginning of the chapter
  Kolmogorov explains that elementary theory of probability
  is defined as part of the theory in which we have to deal
  with probabilities of only a finite number of events.
  It appears that this restriction is applicable only
  to the mathematical theory
  and not to its informal motivation in the \emph{Verh\"altnis}.
  Unlimited repetition means that the number of trials is potentially infinite.
  Fixing the number of repetitions in advance would contradict
  this requirement,
  and even fixing a finite stopping time would contradict it
  (by the finiteness of the number of events considered at each step
  combined with K\H{o}nig's lemma).
  For the purposes of Principles~A and~B, however,
  the numbers of trials appear to be fixed;
  this is definitely true for B (one trial)
  and likely true for A ($n$ chosen in advance).
\end{remark}

The next part derives informally from Principle~A
the axioms of finitely additive probability introduced in the previous section, \cite[Sect.~I.1]{Kolmogoroff:1933}.
For example, if events $A_1$ and $A_2$ are disjoint,
the frequencies $m$, $m_1$, and $m_2$ of occurrences of $A_1\cup A_2$, $A_1$, and $A_2$,
respectively,
in $n$ repetitions will satisfy
\[
  \frac{m}{n}
  =
  \frac{m_1}{n}
  +
  \frac{m_2}{n},
\]
and so ``it appears appropriate to set''
$P(A_1\cup A_2)=P(A_1)+P(A_2)$.
This is the informal derivation of Axiom~V in \cite[Sect.~I.1]{Kolmogoroff:1933};
analogous arguments work for all other axioms.

Having both Principles~A and~B is unusual,
and Kolmogorov's predecessors often assumed only Principle~B,
sometimes referred to as Cournot's principle.
Once we assume Principle~B, Principle~A seems to follow by the law of large numbers,
a mathematical statement (details will be given below).
Such statements were made, e.g., on numerous occasions by Borel and L\'evy
\cite[Sect.~5.2.2]{Shafer/Vovk:2006SS}.
This will be further discussed in Sect.~\ref{sec:why-two}.

The final part of the \emph{Verh\"altnis} contains two remarks.
Remark~I says that if two assertions are both practically certain,
then their conjunction is also practically certain,
though with a little lower degree of certainty;
however, if the number of assertions is very large,
we cannot draw any conclusion whatsoever
about the practical certainty of their conjunction
from the practical certainty of each of them individually.
This is a version of the requirement that the event $A$ in Principle~B
should be prespecified,
and it prevents what philosophers call the lottery paradox
(see \cite{Wheeler:2007} for a review).
Remark~I appears uncontroversial.

Remark~II starts by observing that while the axioms of the previous section \cite[Sect.~I.1]{Kolmogoroff:1933} imply
that the impossible event $\emptyset$ has probability zero,
the converse implication is not true.
The rest of Remark~II claims much more \cite[Sect.~5.2.1]{Shafer/Vovk:2006SS}:
\begin{quote}
  By Principle~B, the event $A$'s having probability zero
  implies only that it is practically impossible
  that it will happen on a particular unrepeated realization of the conditions $\mathfrak{S}$.
  This by no means implies that the event $A$ will not appear
  in the course of a sufficiently long series of experiments.
  When $P(A)=0$ and $n$ is very large,
  we can only say, by Principle A,
  that the quotient $m/n$ will be very small---%
  it might, for example, be equal to $1/n$.
\end{quote}
In the second Russian edition the wording of Remark~II became less emphatic,
with the Russian equivalent of ``by no means''
(\begin{cyr}никоим образом\end{cyr}) removed.
It is interesting that Kolmogorov and Shiryaev also dropped the mention of the ratio $1/n$
as an example of non-zero $m/n$.

Remark~II appears to contradict the modern intuition of probability.
If a prespecified event $A$ pertaining to one trial has zero probability,
observing it in a sequence of trials, finite or infinite,
should also have zero probability, no matter whether the trials are independent or not.
This will be further discussed in Sect.~\ref{sec:Remark-II}.

\section{Why two principles?}
\label{sec:why-two}

Can we really derive Principle~A from Principle~B and a mathematical statement?
Doing so requires applying mathematical probability theory
to a combination of $n$ experiments rather than one experiment.
Strictly speaking, Kolmogorov's conditions $\mathfrak{S}$
are for one experiment, not for a sequence of experiments.
Are the experiments in Principle~A combinable
in such a way that Kolmogorov's empirical picture is applicable
to the combined experiment?
(See, e.g., \cite[Sect.~5.2.2]{Shafer/Vovk:2006SS}.)

In our current context,
there are two basic kinds of combinability:
under A-combinability,
we are allowed to apply Principle~A to the combination,
and under B-combinability,
we are allowed to apply Principle~B.
We can also talk about full combinability,
where the combination is subject to both principles.

B-combinability looks to me the simplest and most natural kind.
This is what we need for the derivation of Principle~A from Principle~B
under the additional assumption that the combined experiment
is described by the probability measure $P^n$,
where $P$ is the probability measure
(perhaps only finitely additive at this point in Kolmogorov's exposition)
describing one trial.
Kolmogorov's proof of the law of large numbers
in the \emph{Grundbegriffe}
\cite[Sect.~VI.3, the part present in all editions]{Kolmogoroff:1933}
implies that the relative frequency of occurrences of $A$ in $n$ trials
deviates from $P(A)$ by more than a given constant $\epsilon>0$
with $P^n$ probability at most
\begin{equation}\label{eq:law}
  \frac{P(A)(1-P(A))}{n\epsilon^2}
  \le
  \frac{1}{4n\epsilon^2}.
\end{equation}
For small $\epsilon$ and large $n\epsilon^2$,
we can derive Principle~A from Principle~B.

\begin{remark}
  A-combinability requires repeating experiments
  that already consist of repeated experiments.
  This might look complicated but appear necessary
  under a frequentist interpretation of the law of large numbers,
  e.g., in the form of \eqref{eq:law}.
  In fact, von Mises takes an opposite approach
  (see, e.g., \cite[p.~114 of the English translation]{vonMises:1928}).
  Instead of repeating infinitely many times
  an experiment repeated finitely many times,
  von Mises starts from the combination,
  a single infinite sequence,
  which is then cut into an infinite sequence of finite sequences of length $n$
  (using his operations of selection and combination).
  We can apply the same approach in the case of repeating finitely many times
  an experiment repeated finitely many times;
  this would be more consistent with Kolmogorov's philosophy and even easier.
\end{remark}

There are several possible explanations
for Kolmogorov having both Principle~A and Principle~B
in his book.
The most basic explanation is that the trials in Kolmogorov's picture
are not even B-combinable.
However, this answer seems to limit severely the applicability of probability theory,
which runs against typical modern intuitions of probability.

Another explanation is that, even if Kolmogorov's experiments are B-combinable,
the probability measure (at this point of the book finitely additive)
does not have to be a product measure;
the component experiments do not have to be independent
or, more generally, do not have to be IID.
This was Shiryaev's implicit suggestion
in \cite[bottom of p.~324]{Shiryaev:2019-local},
and it was discussed more explicitly in \cite[Sect.~5.2]{Shafer/Vovk:2006SS}.

The third explanation is that the accuracy of zero probability diminishes
as we increase the number of trials.
As mentioned in \cite[Sect.~2.2.3]{Shafer/Vovk:2006SS},
there was no gulf between zero probability and merely small probability
in the writings of Borel and his French colleagues;
for them, the vanishingly small was merely an idealization of the very small.
However, extending this to Kolmogorov seems
to run counter to everything we know about his rigour.

To me, the most convincing explanation is that,
even if we regard Principle~A as derivable from Principle~B,
we still need the former to give an operational interpretation of probability.
Principle~A gives us a way of measuring probabilities in principle
(ignoring the unknown accuracy of such measurements)
and allows us to motivate Kolmogorov's axioms
(see also \cite{Bartlett:1949} and \cite[Sect.~5.2.2]{Shafer/Vovk:2006SS}).
In principle,
from the modern philosophical point of view
there is no need to have an operational interpretation
of theoretical notions \cite{Chang:2021},
but it may have been important for Kolmogorov,
with the extra bonus of justifying the first five axioms.
The original title of the \emph{Verh\"altnis}
was ``Empirische Begr\"undung der Axiome''
(Kolmogorov's letter of 26 October 1932 in \cite{Shiryaev:2003b};
essentially this is still the name of a subsection,
``Empirische Deduktion der Axiome''),
which makes the presence of Principle~A especially desirable.
Perhaps having Principle~A separately was also a pedagogical device
employed by Kolmogorov in teaching his Dnipro students.

\section{A finitely additive measure consistent with Remark~II}
\label{sec:Remark-II}

The assumption of B-combinability appears as natural
in the context of Remark~II as in the context of deriving Principle~A from Principle~B
in the previous section.
On the one hand,
the number of trials in Remark~II is potentially infinite,
as discussed in Remark~\ref{rem:infinity},
on the other hand, we do not need any assumptions
on the dependence structure of the sequence of trials
(remember that the notion of independence,
the topic of \cite[Sect.~I.5]{Kolmogoroff:1933},
is not even introduced yet when discussing Principles~A and~B
in the \emph{Grundbegriffe}).
If the potentially infinite sequence of experiments
in Remark~II is not B-combinable,
the resulting theory appears very restrictive;
why shouldn't it be applicable to a sequence of experiments?

If the complex of conditions $\mathfrak{S}$ is repeated
a given finite number $n$ of times,
Remark~II is not applicable,
since the five axioms introduced in Sect.~I.1 imply
\[
  P(A_1\cup\dots\cup A_n)
  \le
  \sum_{i=1}^n
  P(A_i) = 0
\]
regardless of whether or not the trials are dependent,
where $A_i$ is the event (in the combined experiment)
of observing the event $A$ in the $i$th trial.
The conclusion is also true when we fix in advance a finite stopping time,
as discussed in Remark~\ref{rem:infinity}.
Even if the sequence of trials is infinitely long,
Remark~II still fails under countable additivity.
Perhaps this demonstrates Kolmogorov's contention in Chapter II
that ``it is almost impossible to elucidate''
the empirical meaning of countable additivity.

Let us check that Remark~II is compatible with the first five axioms,
those that do have a clear empirical meaning according to Kolmogorov.
First we replace a ``potentially infinite'' sequence of trials
with an infinite sequence;
this seems inevitable in mathematical considerations,
and potential infinity will enter via our interest exclusively in events
being settled after finitely many trials.
We need infinite sequences only because we do not have
an everywhere finite stopping time chosen in advance.

Let us encode the appearance of the event $A$ of probability zero as 1
and its failure to appear as 0.
This makes our sample space $\Omega:=\{0,1\}^{\infty}$;
define $X_n(\omega)$, $\omega\in\Omega$, as the $n$th element of $\omega$.
Let $\N:=\{1,2,\dots\}$.
The following simple corollary of known results is a possible
mathematical counterpart of Remark~II with the individual trials
regarded as B-combinable.

\begin{proposition}\label{prop:existence}
  There exists a finitely additive probability measure $P$
  defined on all subsets of $\Omega$
  such that $P(X_n=1)=0$ for each $n\in\N$
  whereas $P(\exists n\in\N:X_n=1)=1$.
\end{proposition}

\begin{proof}[Proof sketch]
  This follows immediately
  from Theorem 1 in \cite{Kadane/OHagan:1995}
  (stated in the appendix as Theorem~\ref{thm:KO}),
  which gives a simple necessary and sufficient condition for the existence
  of an extension of a function defined on some subsets of $\Omega$
  to all subsets of $\Omega$.
\end{proof}

In the appendix we will see what such a probability measure $P$ may look like.
It is impossible to really construct $P$,
since its existence depends on the axiom of choice
(or something similarly nonconstructive),
but we can easily construct it when given a simple mathematical object
(a free ultrafilter on $\N$).

When observing the trials sequentially,
we can stop when the $A$ of Remark~II happens
(by Principle~B this is practically certain under $P$),
and then the ratio $m/n$ will be $1/n$.
While Axiom~VI, introduced later, in Sect.~II.1,
makes the appearance of $A$ practically impossible,
without it we can make the appearance of $A$ inevitable.
Let us state formally this and other properties of $P$.

\begin{proposition}\label{prop:tau}
  Let $P$ be any finitely additive probability measure as in Proposition~\ref{prop:existence}
  and let $\tau(\omega):=\min\{n\st\omega_n=1\}$ (with $\min\emptyset:=\infty$).
  Then:
  \begin{enumerate}[(a)]
  \item
    $P$ agrees with the point mass at $(0,0,\dots)$
    on every event determined by finitely many trials;
    in particular, $X_1,X_2,\dots$ are independent
    (in the finitely additive sense) under $P$.
  \item
    $P(\tau>n)=1$ for every $n\in\N$, whereas $P(\tau<\infty)=1$.
  \end{enumerate}
\end{proposition}

The proof is spelled out in detail in the appendix.
Part (a) says that Remark II is not being bought at the price of exotic dependence:
the trials may be taken independent,
and every prediction about a fixed finite horizon is exactly the one
the classical model makes.
Part (b) gives Kolmogorov's $1/n$:
at the stopping time $\tau$,
which is $P$-almost surely finite but larger than every fixed $n$
with probability one,
the observed frequency is exactly $1/\tau$.

\section{Conclusion}

This note reflects my attempts to reconcile Kolmogorov's thinking
with modern views,
and it pays particular attention to the aspects of the \emph{Verh\"altnis}
that have looked puzzling to me.
For a broader perspective, the reader may consult
von Plato \cite{vonPlato:1998} and Shafer and Vovk \cite{Shafer/Vovk:2006SS}.

In his preface to the second Russian edition of the \emph{Grundbegriffe}
Kolmogorov points out a fundamental limitation of the \emph{Verh\"altnis}
and draws the reader's attention to the new approach
to the relation of probability theory to reality
sketched in his papers \cite{Kolmogorov:1965} and \cite{Kolmogorov:1968}.
But this is a different story.

\subsection*{Acknowledgments}

Many thanks to Glenn Shafer for numerous discussions of the \emph{Grundbegriffe}
including \cite[Sect.~I.2, including Remark~II]{Kolmogoroff:1933}
and to Sasha Shen for a more recent discussion of \cite[Sect.~I.2, Remark~II]{Kolmogoroff:1933}.
In this note I used Shen's comparison of the wordings of Remark~II in the \emph{Verh\"altnis}
in different Russian editions.

I acknowledge the use of Claude Opus 5 in exploring proof ideas,
which I reviewed carefully,
and for checking the note.
I take full responsibility for this note's claims and statements,
including mathematical statements and their proofs.

\appendix
\section{Some details}

In this appendix I will give a more intuitive argument
for Proposition~\ref{prop:existence}
that is also more explicit
(it is constructive modulo a given free ultrafilter $\UUU$ on $\N$,
although such a $\UUU$ cannot be constructed itself,
as mentioned earlier).
In order to make the exposition self-contained,
its second section states Theorem~1 of \cite{Kadane/OHagan:1995}
and gives the details of the derivation of Proposition~\ref{prop:existence}.
Finally, a detailed proof of Proposition~\ref{prop:tau} is also provided.

\subsection{An explicit $P$ in Proposition~\ref{prop:existence}}

The notion of a \emph{free ultrafilter} (on $\N$) is very intuitive,
but let us start from a formal definition.
We say that a family $\UUU$ of subsets of $\N$ is a \emph{filter}
if $\emptyset\notin\UUU$, $\N\in\UUU$,
a superset of any element of $\UUU$ is in $\UUU$,
and the intersection of any two elements in $\UUU$ is in $\UUU$.
A filter $\UUU$ is an \emph{ultrafilter}
if for every $A\subseteq\N$, either $A$ or its complement is in $\UUU$.
An ultrafilter $\UUU$ is \emph{free} if it does not contain any singleton.

Fix a free ultrafilter $\UUU$ on $\N:=\{1,2,\dots\}$.
Let us say that a set $A\subseteq\N$ is \emph{big} if $A\in\UUU$;
otherwise, it is \emph{small}.
The definition of a free ultrafilter requires that every singleton be a small set,
subsets of small sets be small,
finite unions of small sets be small,
and the complements of small sets be big and vice versa.

Let $e_k$, $k\in\N$, be the sequence in $\Omega=\{0,1\}^{\infty}$
that has a 1 only at the $k$th position.
Define $P$ to be the finitely additive measure concentrated on $\{e_k\st k\in\N\}$
such that, for any $E\subseteq\Omega$,
\begin{equation}\label{eq:P}
  P(E)
  :=
  \begin{cases}
    1 & \text{if $\{k\st e_k\in E\}\in\UUU$}\\
    0 & \text{otherwise}.
  \end{cases}
\end{equation}
In other words, $P(E)=1$ if $E$ contains many $e_k$
(in the sense of $\{k\st e_k\in E\}$ being big)
and $P(E)=0$ if $E$ contains few $e_k$
(in the sense of $\{k\st e_k\in E\}$ being small).
The intuition behind $P$ is that it expresses the belief
that the sequence of outcomes will be one of the $e_k$
(i.e., will contain only one 1)
and that we believe that $k\notin A$ for a prespecified small $A\subseteq\N$
while we believe that $k\in A$ for a prespecified big $A\subseteq\N$.
It is obvious that $P$ satisfies the conditions of Proposition~\ref{prop:existence}.

The finitely additive probability measure~\eqref{eq:P} is not very interesting
in that it takes only two values, $0$ or $1$.
To make its range $[0,1]$, we can modify \eqref{eq:P}
using the notion of the ``standard part'' of a sequence of real numbers;
before defining it formally, let me explain how it can be used in constructing $P$.
The standard part of a sequence is its generalized limit;
it always exists in $\overline{\R}$ and coincides with the limit when it exists.
We define $P(E)$ as the standard part of the sequence $(p_1,p_2,\dots)$, where
\begin{equation}\label{eq:p}
  p_n
  :=
  \left|
    \{k\le n\st e_k\in E\}
  \right|
  /n
\end{equation}
is the relative frequency of $e_k\in E$ over the first $n$ steps.
We can say that $P(E)$ is the limiting relative frequency of $e_k\in E$, $k\in\N$.

Let us now define $P$ formally.
A \emph{hyperreal (number)} is defined to be a sequence of real numbers,
where hyperreal numbers $(a_1,a_2,\dots)$ and $(b_1,b_2,\dots)$
are regarded as equivalent if $a_k=b_k$ for many $k$
(formally, a hyperreal number can be also defined as an equivalence class).
The linear order on the hyperreals is defined in the same way:
$(a_1,a_2,\dots)\le(b_1,b_2,\dots)$ means that
$a_k\le b_k$ for many $k$.
Real numbers can be embedded into the hyperreals as
$a\in\R\mapsto(a,a,\dots)$.
It is clear that for any hyperreal $p=(p_1,p_2,\dots)$
with all $p_k\in[0,1]$ (such as \eqref{eq:p})
there exists a unique real number $p'\in[0,1]$
such that, for any real $\epsilon>0$, $\left|p-p'\right|<\epsilon$
(where the operations $\left|\cdot\right|$ and $-$ are defined
for hyperreals term-wise).
Such $p'$ is known as the \emph{standard part} of $p$.
This completes the definition of $P$;
checking that it satisfies the conditions of Proposition~\ref{prop:existence}
is straightforward.

An excellent source about ultrafilters
(and Banach limits, which can be used in place of taking the standard part)
is \cite{Komjath/Totik:2008}.
For a short self-contained description of ultrafilters and hyperreals,
see the first subsection of Sect.~11.5 in the book \cite{Shafer/Vovk:2001},
which applies them in the foundations of probability.
For a detailed exposition of non-standard analysis (including the hyperreals)
based on ultrafilters,
see \cite{Goldblatt:1998}.

\subsection{Details of the deduction of Proposition~\ref{prop:existence}
  from a general criterion}

Proposition~\ref{prop:existence} follows from the following
necessary and sufficient condition for the possibility of extending a given function
on subsets of $\Omega$ to all subsets.

\begin{theorem}[\cite{Kadane/OHagan:1995}, Theorem~1]\label{thm:KO}
  Let $\mathcal{C}$ be any collection of subsets of a set $\Omega$
  with $\Omega\in\mathcal{C}$,
  and let $P$ be a nonnegative real-valued function on $\mathcal{C}$ with $P(\Omega)=1$.
  Then $P$ can be extended to a finitely additive probability measure on all subsets of $\Omega$
  if and only if, for all finite lists $C_{1},\dots,C_{c}$ and $C'_{1},\dots,C'_{c'}$
  of members of~$\mathcal{C}$,
  \begin{equation}\label{eq:thm}
    \left(
      \sum_{i=1}^{c} 1_{C_{i}} \le \sum_{j=1}^{c'} 1_{C'_{j}}
    \right)
    \Longrightarrow
    \left(
      \sum_{i=1}^{c} P(C_{i}) \le \sum_{j=1}^{c'} P(C'_{j})
    \right),
  \end{equation}
  where $1_{A}$ stands for the indicator function of $A$.
\end{theorem}

The direction that we need for Proposition~\ref{prop:existence}
(namely, the ``if'' direction)
is actually stated already in \cite[Theorem~3.2.10]{Rao:1983}.

Let us now deduce Proposition~\ref{prop:existence} from Theorem~\ref{thm:KO}.
Let $\CCC$ consist of the sets $\Omega$, $Z:=\{(0,0,\dots)\}$,
and $\{X_n=1\}$, $n\in\N$.
It suffices to check \eqref{eq:thm},
where $P(\Omega):=1$, $P(Z):=0$, and as before, $P(\{X_n=1\})=0$ for all $n\in\N$.
Arguing indirectly, suppose \eqref{eq:thm} is violated for some lists.
Without loss of generality, we can assume that
all $C'_j$ are either $Z$ or of the form $\{X_n=1\}$
and that $c=1$ and $C_1=\Omega$.
We arrive at a contradiction by noticing
that the antecedent of \eqref{eq:thm} is false at $e_i$
for sufficiently large $i$.

\subsection{Proof of Proposition~\ref{prop:tau}}

(a)
Let $E$ be determined by the trials in a finite set $F\subseteq\N$,
and let $C:=\{\omega\st\omega_n=0 \text{ for all } n\in F\}$.
The complement of $C$ is $\bigcup_{n\in F}\{X_n=1\}$,
a finite union of events of probability zero,
so $P(C)=1$.
If $(0,0,\dots)\in E$ then $C\subseteq E$ and $P(E)=1$;
otherwise $C\subseteq E^{\textrm{c}}$ and $P(E)=0$.
Independence follows since every finite-dimensional cylinder
is such an $E$, and its probability, 1 or 0 according as
it requires no 1 or some 1, agrees with the product of the marginals.

(b)
$\{\tau>n\}=\{X_1=\dots=X_n=0\}$ has probability 1 by~(a),
and $\{\tau<\infty\}=\{\exists n\in\N\st X_n=1\}$
has probability 1 by assumption.

\begin{thebibliography}{10}
\bibitem{Bartlett:1949}
Maurice~S. Bartlett.
\newblock Probability in logic, mathematics and science.
\newblock {\em Dialectica}, 3:104--113, 1949.

\bibitem{Rao:1983}
K.~P.~S. Bhaskara~Rao and M.~Bhaskara~Rao.
\newblock {\em Theory of Charges}.
\newblock Academic Press, New York, 1983.

\bibitem{Chang:2021}
Hasok Chang.
\newblock Operationalism.
\newblock In Edward~N. Zalta, editor, {\em Stanford Encyclopedia of
  Philosophy}. Metaphysics Research Lab, Stanford University, {Autumn} 2021
  edition, 2021.

\bibitem{Goldblatt:1998}
Robert Goldblatt.
\newblock {\em Lectures on the Hyperreals: An Introduction to Nonstandard
  Analysis}.
\newblock Springer, New York, 1998.

\bibitem{Hajek:2023}
Alan H\'ajek.
\newblock Interpretations of probability.
\newblock In Edward~N. Zalta and Uri Nodelman, editors, {\em Stanford
  Encyclopedia of Philosophy}. Metaphysics Research Lab, Stanford University,
  {Winter} 2023 edition, 2023.

\bibitem{Kadane/OHagan:1995}
Joseph~B. Kadane and Anthony O'Hagan.
\newblock Using finitely additive probability: uniform distributions on the
  natural numbers.
\newblock {\em Journal of the American Statistical Association}, 90:626--631,
  1995.

\bibitem{Kolmogoroff:1933}
Andrei~N. Kolmogorov.
\newblock {\em Grundbegriffe der Wahr\-schein\-lich\-keits\-rechnung}.
\newblock Springer, Berlin, 1933.
\newblock Russian translation: \begin{cyr}Основные понятия
  теории вероятностей\end{cyr}. \begin{cyr}ОНТИ\end{cyr},
  Moscow, 1936. {English} translation: \emph{Foundations of the Theory of
  Probability}. Chelsea, New York, 1950. Second English edition: Chelsea, New
  York, 1956. Second Russian edition: \begin{cyr}Наука\end{cyr}, Moscow,
  1974. Third Russian edition: \begin{cyr}ФАЗИС\end{cyr}, Moscow, 1998.

\bibitem{Kolmogorov:1965}
Andrei~N. Kolmogorov.
\newblock Three approaches to the quantitative definition of information.
\newblock {\em Problems of Information Transmission}, 1:1--7, 1965.
\newblock {Russian} original: \begin{cyr}Три под\-хо\-да к
  опре\-де\-ле\-нию понятия ``количество
  ин\-фор\-ма\-ции''\end{cyr}.

\bibitem{Kolmogorov:1968}
Andrei~N. Kolmogorov.
\newblock Logical basis for information theory and probability theory.
\newblock {\em IEEE Transactions on Information Theory}, IT-14:662--664, 1968.
\newblock {Russian} version: \begin{cyr}К логическим основам
  теории информации и теории
  ве\-ро\-ят\-нос\-тей\end{cyr}, published in
  \begin{cyr}Проблемы передачи информации\end{cyr},
  1969.

\bibitem{Komjath/Totik:2008}
Péter Komjáth and Vilmos Totik.
\newblock Ultrafilters.
\newblock {\em American Mathematical Monthly}, 115:33--44, 2008.

\bibitem{Nikolsky:1983}
Sergei~M. Nikol'skii.
\newblock Aleksandrov and {Kolmogorov} in {Dnepropetrovsk}.
\newblock {\em Russian Mathematical Surveys}, 38(4):41--55, 1983.

\bibitem{Nikolsky:2000}
Sergei~M. Nikol'skii.
\newblock In memory of {A.~N.~Kolmogorov}.
\newblock In {\em Kolmogorov in Perspective}, pages 109--116. American
  Mathematical Society, 2000.

\bibitem{Shafer/Vovk:2001}
Glenn Shafer and Vladimir Vovk.
\newblock {\em Probability and Finance: It's Only a Game!}
\newblock Wiley, New York, 2001.

\bibitem{Shafer/Vovk:2006SS}
Glenn Shafer and Vladimir Vovk.
\newblock The sources of {Kolmogorov's} \emph{Grundbegriffe}.
\newblock {\em Statistical Science}, 21:70--98, 2006.

\bibitem{Shiryaev:2003a}
Albert~N. Shiryaev, editor.
\newblock {\em \noop{a}Truth is a blessing: Biobibliography (in Russian,
  \begin{cyr}Истина-благо:
  Биобиблиография\end{cyr})}, volume~1 of {\em
  \begin{cyr}Колмогоров: юбилейное издание в
  трех книгах\end{cyr}}.
\newblock \begin{cyr}Физ\-мат\-лит\end{cyr}, Moscow, 2003\noop{1}.

\bibitem{Shiryaev:2003b}
Albert~N. Shiryaev, editor.
\newblock {\em \noop{b}The braid of these running lines\dots: Selected passages
  from the correspondence between A.~N.~Kolmogorov and P.~S.~Aleksandrov (in
  Russian, \begin{cyr}Этих строк бегущих тесьма\dots:
  Избранные места из пе\-ре\-пис\-ки
  А.~Н.~Колмогорова и
  П.~С.~Александрова\end{cyr})}, volume~2 of {\em
  \begin{cyr}Кол\-мо\-го\-ров: юбилейное издание в
  трех книгах\end{cyr}}.
\newblock \begin{cyr}Физ\-мат\-лит\end{cyr}, Moscow, 2003\noop{2}.

\bibitem{Shiryaev:2019-local}
Albert~N. Shiryaev.
\newblock Development of mathematical theory of probability: Historical review.
\newblock In {\em Albert N. Shiryaev, Probability-2}, pages 313--331. Springer,
  New York, 2019.
\newblock A version of this essay was published as
  ``\begin{cyr}Математическая теория
  ве\-ро\-ят\-нос\-тей.\ Очерк истории
  становления\end{cyr}'' in the third Russian edition of
  Kolmogorov's \emph{Grundbegriffe} \cite{Kolmogoroff:1933}.

\bibitem{vonMises:1928}
Richard von Mises.
\newblock {\em Wahrscheinlichkeit, Statistik und Wahrheit}.
\newblock Springer, Berlin, 1928.
\newblock {English} translation: \emph{Probability, Statistics and Truth}.
  William Hodge, London, 1939.

\bibitem{vonMises:1931}
Richard von Mises.
\newblock {\em Wahrscheinlichkeitsrechnung und ihre Anwendung in der Statistik
  und theoretischen Physik (in German, Probabilities and their Applications in
  Statistics and Theoretical Physics)}, volume~1 of {\em Vorlesungen aus dem
  {Gebiete} der angewandten Mathematik (Lectures on Applied Mathematics)}.
\newblock Franz Deuticke, Leipzig and Vienna, 1931.
\newblock This book was the only volume in the series \emph{Lectures on Applied
  Mathematics}.

\bibitem{vonPlato:1998}
Jan von Plato.
\newblock {\em Creating Modern Probability: Its Mathematics, Physics and
  Philosophy in Historical Perspective}.
\newblock Cambridge University Press, Cambridge, 1994.

\bibitem{Wheeler:2007}
Gregory Wheeler.
\newblock A review of the lottery paradox.
\newblock In Gregory Wheeler and William Harper, editors, {\em Probability and
  Inference: Essays in Honour of Henry E. Kyburg Jr.}, pages 1--31. College
  Publications, London, 2007.
\end{thebibliography}
\end{document}